\documentclass[12pt]{amsart}

\usepackage{amsmath,amssymb,amsfonts,amsthm,latexsym,graphicx,multirow}
\usepackage{hyperref}

      \newcommand\Cr{\mathrm{cr}}

  \newcommand\fix{\mathrm{fix}}

\newcommand\id{\mathrm{id}}

  \newcommand\PGL{\mathrm{PGL}}

       \newcommand\Sy{\mathrm{S}}

\newtheorem{theorem}{Theorem}[section]
\newtheorem{lemma}[theorem]{Lemma}

\theoremstyle{definition}

\begin{document}

\title{The covering radius of $\PGL_2(q)$}

\author[Xia]{Binzhou Xia}
\address{School of Mathematics and Statistics\\University of Western Australia\\ Crawley 6009, WA\\ Australia}
\email{binzhou.xia@uwa.edu.au}

\maketitle

\begin{abstract}
The covering radius of a subset $C$ of the symmetric group $\Sy_n$ is the maximal Hamming distance of an element of $\Sy_n$ from $C$. This note determines the covering radii of the finite $2$-dimensional projective general linear groups. It turns out that the covering radius of $\PGL_2(q)$ is $q-2$ if $q$ is even, and is $q-3$ if $q$ is odd.

\end{abstract}

\section{Introduction}

Let $n\geqslant2$. Define the \emph{Hamming distance} $d$ on the symmetric group $\Sy_n$ by letting
\[
d(g,h)=n-|\fix(gh^{-1})|
\]
for any $g,h\in\Sy_n$, where $\fix$ denotes the set of fixed points. Note that
\[
|\fix(gh^{-1})|=|\fix(hg^{-1})|=|\fix(g^{-1}h)|=|\fix(h^{-1}g)|.
\]
With Hamming distance, $\Sy_n$ is a metric space. Then the distance $d(v,C)$ of a point $v$ from a subset $C$ in $\Sy_n$ is $\min\{d(v,c)\mid c\in C\}$, and the \emph{covering radius} of $C$ is
\[
\Cr(C)=\max\{d(v,C)\mid v\in\Sy_n\}.
\]
Covering radii of subgroups of $\Sy_n$ were studied by Cameron and Wanless in~\cite{CW2005}, among other things, where particular interest was in the subgroup $\PGL_2(q)$ of $\Sy_{q+1}$ with prime power $q$. They proved:

\begin{theorem}\label{thm1}
\emph{(\cite[Theorem~22]{CW2005})} If $q\not\equiv1\pmod{6}$, then
\begin{equation}\label{eq2}
\Cr(\PGL_2(q))=
\begin{cases}
q-2,\quad\text{if $q$ is even,}\\
q-3,\quad\text{if $q$ is odd}.
\end{cases}
\end{equation}
If $q\equiv1\pmod{6}$, then $q-5\leqslant \Cr(\PGL_2(q))\leqslant q-3$.
\end{theorem}

In this note, we resolve the case $q\equiv1\pmod{6}$ in Theorem~\ref{thm1} by proving:

\begin{theorem}\label{thm2}
If $q\equiv1\pmod{6}$, then $\Cr(\PGL_2(q))=q-3$.
\end{theorem}

Combining Theorems~\ref{thm1} and~\ref{thm2} one sees that indeed~\eqref{eq2} holds for all prime power $q$. This completely determines the covering radii of finite $2$-dimensional projective general linear groups.

\section{Proof of Theorem~\ref{thm2}}

Let $q$ be a prime power such that $q\equiv1\pmod{6}$, let $G=\PGL_2(q)$ acting on $\Omega:=\mathbb{F}_q\cup\{\infty\}$, and let $\Delta=\{y\in\mathbb{F}_{q^2}\mid y^{q+1}=-1\}$. As $q$ is odd, $|\mathbb{F}_{q^2}^\times|=q^2-1$ is divisible by $2(q+1)$. Take $\rho$ to be an element of $\mathbb{F}_{q^2}^\times$ of order $2(q+1)$.

\begin{lemma}\label{lem1}
$\rho^{q+1}=-1$ and $\rho\notin\mathbb{F}_q$.
\end{lemma}

\begin{proof}
Since $\rho$ has order $2(q+1)$, we have $\rho^{2(q+1)}=1$ and $\rho^{q+1}\neq1$. This gives that $(\rho^{q+1}+1)(\rho^{q+1}-1)=0$ while $\rho^{q+1}-1\neq0$. Thus, $\rho^{q+1}+1=0$, i.e. $\rho^{q+1}=-1$. Moreover, as $\rho$ has order $2(q+1)>q-1=|\mathbb{F}_q^\times|$ we know that $\rho\notin\mathbb{F}_q$.
\end{proof}

For any $x\in\mathbb{F}_q$ and $y\in\Delta\setminus\{-1/\rho\}$, let
\[
x^\sigma=\frac{x+\rho}{1-\rho x}\quad\text{and}\quad y^\tau=\frac{y-\rho}{1+\rho y}.
\]
Note that $1-\rho x\neq0$ for any $x\in\mathbb{F}_q$ since $\rho\notin\mathbb{F}_q$ by Lemma~\ref{lem1}. As in the usual convention, let
\begin{equation}\label{eq1}
\infty^\sigma=-\frac{1}{\rho}\quad\text{and}\quad\left(-\frac{1}{\rho}\right)^\tau=\infty.
\end{equation}

\begin{lemma}\label{lem2}
$\sigma$ is a map from $\Omega$ to $\Delta$ and $\tau$ is a map from $\Delta$ to $\Omega$ such that $\sigma\tau=\id_\Omega$ and $\tau\sigma=\id_\Delta$.
\end{lemma}

\begin{proof}
According to Lemma~\ref{lem1}, $\rho^{q+1}=-1$. Then for any $x\in\mathbb{F}_q$, since $x^q=x$, we have
\begin{eqnarray*}
\left(x^\sigma\right)^{q+1}&=&\left(\frac{x+\rho}{1-\rho x}\right)^{q+1}\\
&=&\left(\frac{x+\rho}{1-\rho x}\right)^q\frac{x+\rho}{1-\rho x}\\
&=&\frac{x+\rho^q}{1-\rho^q x}\cdot\frac{x+\rho}{1-\rho x}\\
&=&\frac{\rho x+\rho^{q+1}}{\rho-\rho^{q+1}x}\cdot\frac{x+\rho}{1-\rho x}\\
&=&\frac{\rho x-1}{\rho+x}\cdot\frac{x+\rho}{1-\rho x}=-1.
\end{eqnarray*}
Also, $(\infty^\sigma)^{q+1}=(-1/\rho)^{q+1}=1/\rho^{q+1}=-1$. This shows that $\sigma$ is a map from $\Omega$ to $\Delta$. For any $y\in\Delta\setminus\{-\rho^{-1}\}$, in view of $y^{q+1}=\rho^{q+1}=-1$ we deduce that
\[
\left(y^\tau\right)^q=\left(\frac{y-\rho}{1+\rho y}\right)^q=\frac{y^q-\rho^q}{1+\rho^q y^q}
=\frac{\rho y^{q+1}-\rho^{q+1}y}{\rho y+\rho^{q+1}y^{q+1}}=\frac{-\rho+y}{\rho y+1}=y^\tau
\]
and hence $y^\tau\in\mathbb{F}_q$. Then as $(-1/\rho)^\tau=\infty$, we see that $\tau$ is a map from $\Delta$ to $\Omega$. Finally, for any $x\in\mathbb{F}_q$ and $y\in\Delta\setminus\{-1/\rho\}$,
\[
y=\frac{x+\rho}{1-\rho x}\quad\text{if and only if}\quad x=\frac{y-\rho}{1+\rho y}.
\]
This in conjunction with~\eqref{eq1} implies that for any $x\in\Omega$ and $y\in\Delta$,
\[
y=x^\sigma\quad\text{if and only if}\quad x=y^\tau.
\]
As a consequence, we obtain $\sigma\tau=\id_\Omega$ and $\tau\sigma=\id_\Delta$.
\end{proof}

For any $y\in\Delta$, let $y^h=y^3$.

\begin{lemma}\label{lem3}
$h$ is a permutation on $\Delta$.
\end{lemma}

\begin{proof}
Clearly, $y^3\in\Delta$ for any $y\in\Delta$. Hence $h$ is a map from $\Delta$ to $\Delta$. If $y_1\in\Delta$ and $y_2\in\Delta$ such that $y_1^3=y_2^3$, then as $q\equiv1\pmod{3}$,
\[
y_1=\frac{y_1^{q+2}}{y_1^{q+1}}=\frac{(y_1^3)^{(q+2)/3}}{-1}=\frac{(y_2^3)^{(q+2)/3}}{-1}=\frac{y_2^{q+2}}{y_2^{q+1}}=y_2.
\]
Consequently, $h$ is a permutation on $\Delta$.
\end{proof}

From Lemmas~\ref{lem2} and~\ref{lem3} we deduce that $\sigma h\tau$ is a permutation on $\Omega$. In the following we prove $d(\sigma h\tau,G)\geqslant q-3$. Let $g$ be an arbitrary element of $G$. Then $g$ is a linear fractional transformation, and so is $\tau g\sigma$. Accordingly, the equation $y^3=y^{\tau g\sigma}$ on $y$ has at most $4$ solutions over $\mathbb{F}_{q^2}$. In particular,
\[
|\{y\in\Delta\mid y^h=y^{\tau g\sigma}\}|\leqslant4
\]
as $y^h=y^3$. Then by Lemma~\ref{lem2}, it follows that
\begin{eqnarray*}
|\{x\in\Omega\mid x^{\sigma h\tau}=x^g\}|&=&|\{x\in\Omega\mid x^{\sigma h\tau}=x^{\sigma\tau g}\}|\\
&=&|\{x\in\Omega\mid x^{\sigma h\tau}=x^{\sigma\tau g}\}^\sigma|\\
&=&|\{y\in\Delta\mid y^{h\tau}=y^{\tau g}\}|\\
&=&|\{y\in\Delta\mid y^{h\tau\sigma}=y^{\tau g\sigma}\}|\\
&=&|\{y\in\Delta\mid y^h=y^{\tau g\sigma}\}|\leqslant4,
\end{eqnarray*}
which yields
\begin{eqnarray*}
d(\sigma h\tau,g)&=&|\Omega|-|\fix(\sigma h\tau g^{-1})|\\
&=&|\Omega|-|\{x\in\Omega\mid x^{\sigma h\tau}=x^g\}|\geqslant|\Omega|-4=q-3.
\end{eqnarray*}
Therefore, $d(\sigma h\tau,G)\geqslant q-3$ as desired.

Now as there exists a permutation on $\Omega$ at distance at least $q-3$ from $G$, we derive that $\Cr(G)\geqslant q-3$. This together with the inequality $\Cr(G)\leqslant q-3$ given in Theorem~\ref{thm1} leads to $\Cr(G)=q-3$, completing the proof of Theorem~\ref{thm2}.

\vskip0.1in
\noindent\textsc{Acknowledgements.} The author was supported by Australian Research Council grant DP150101066. This note is in response to a problem posed by Peter Cameron and Ian Wanless during their visit to the University of Western Australia. The author is very grateful to them for bringing this problem into his attention.

\end{document}